\documentclass[12pt]{amsart}
\usepackage{amsmath,amsthm,amsrefs}
\usepackage{amsbsy,amsfonts,amssymb}
\usepackage{amsaddr}

\def\n{\noindent}
\def\spo{spo(2m,n)}
\def\ov{\overline}
\def\c{\circ}

\def\ov{\overline}

\theoremstyle{plain}
\newtheorem{thm}{Theorem}[section]
\newtheorem{lem}[thm]{Lemma}

\newtheorem{cor}[thm]{Corollary}

\theoremstyle{definition}

\newtheorem{exa}[thm]{Example}

\newdimen\Squaresize \Squaresize=14pt
\newdimen\Thickness \Thickness=0.4pt

\def\Square#1{\hbox{\vrule width \Thickness
   \vbox to \Squaresize{\hrule height \Thickness\vss
      \hbox to \Squaresize{\hss#1\hss}
   \vss\hrule height\Thickness}
\unskip\vrule width \Thickness} \kern-\Thickness}

\def\Vsquare#1{\vbox{\Square{$#1$}}\kern-\Thickness}

\def\younglambda#1{
\vbox{\smallskip\offinterlineskip \halign{&\Vsquare{##}\cr #1}}}
\def\vy#1{\hskip2pt\vcenter{\younglambda{#1}}\hskip2pt}

\newcommand{\shiftbar}[2]{\ensuremath \raisebox{#1cm}{$\leftarrow\overline{#2}$}}
\newcommand{\shiftnobar}[2]{\ensuremath \raisebox{#1cm}{$\leftarrow{#2}$}}

\title[An orthosymplectic Pieri rule]%
{An orthosymplectic Pieri rule}

\author{Anna Stokke}
\address{University of Winnipeg \\
Department of Mathematics and Statistics \\
Winnipeg, Manitoba \\
Canada  R3B 2E9}

\email{\tt a.stokke@uwinnipeg.ca}
\thanks{This research was supported in part by a grant from the Natural
Sciences and Engineering Research Council of Canada.}

\begin{document}

\begin{abstract} The classical Pieri formula gives a combinatorial rule for decomposing the product of a Schur function and a complete homogeneous
symmetric polynomial as a linear combination of Schur functions with integer coefficients.  We give a Pieri rule for describing the product of an orthosymplectic character and an orthosymplectic character arising from a one-row partition. We establish that the orthosymplectic Pieri rule coincides with Sundaram's Pieri rule for symplectic characters and that orthosymplectic characters and symplectic characters obey the same product rule.  
  \end{abstract}

\maketitle

\section{Introduction}

The celebrated Littlewood-Richardson rule gives a combinatorial description of the coefficients that appear when the product  $s_\lambda s_\mu$ of two Schur functions is written as a linear combination of Schur functions in the ring of symmetric polynomials \cite{littlewood,schutz}.  In the case where one of the partitions is a one-row partition $\mu=(k)$, the expansion of the product takes an especially nice form and the combinatorial description of the coefficients is known as Pieri's rule.  Pieri's rule states that the product $s_\lambda s_{(k)}$ is the sum of Schur functions $s_\nu$, where the Young diagram of shape $\nu$ can be obtained from that of $\lambda$ by 
adding $k$ boxes, where no two of the added boxes are in the same column. 

Berele and Regev \cite{bereleregev} introduced hook Schur functions, which are characters of general linear Lie superalgebras.  Remmel \cite{remmel} proved a Jacobi-Trudi identity for hook Schur functions, via lattice paths, that coincides with that for ordinary Schur functions and also proved that hook Schur functions and ordinary Schur functions satisfy the same Pieri rule, thus establishing that hook Schur functions obey the Littlewood-Richardson rule.  

Symplectic Schur functions,  which are characters of irreducible representations of symplectic groups, can be described in terms of the semistandard symplectic tableaux of King~\cite{king}.  Berele \cite{berele} gave an insertion algorithm for inserting a letter into a symplectic tableau using a combination of Knuth-Schensted row insertion and the jeu de taquin algorithm.  Sundaram \cite{sundaram} gave a Pieri rule for decomposing the product $sp_\lambda sp_{(k)}$ of two symplectic Schur functions, one of which is associated to a one-row partition, as a linear combination of symplectic Schur functions using Berele's row-insertion algorithm.

Benkart, Shader and Ram \cite{benkart} gave combinatorial descriptions of characters of representations of orthosymplectic Lie superalgebras $\spo$, or orthosymplectic Schur functions,   in terms of $spo$-tableaux,
which have both a symplectic part and a row-strict part.  An $spo$-insertion algorithm for inserting a letter into an $spo$-tableau to produce a new $spo$-tableau is also described in \cite{benkart}.

The aim of this paper is to present a Pieri rule  for orthosymplectic Schur functions using $spo$-tableaux. Further, we prove that the orthosymplectic Pieri rule coincides with that for symplectic Schur functions.  In \cite{benkart}, a Jacobi-Trudi identity for orthosymplectic Schur functions is given that mirrors  the symplectic Jacobi-Trudi identity; this result is proved using lattice path arguments in \cite{stokkevisentin}.  It follows that the coefficients arising in the expansion of the product of two orthosymplectic Schur functions are the same as those that appear in the expansion of the product of the corresponding symplectic Schur functions.  

We begin with a preliminary section before discussing the $spo$-insertion algorithm in Section~\ref{sec:insertion}.  We then present a definition for the combinatorial product of an $spo$-tableau and a one-row $spo$-tableau and set about proving several technical lemmas  in Section \ref{sec:main}.  Our main result, Theorem \ref{mainthm}, describes the integer coefficients of the orthosymplectic characters in the expansion of the product $spo_\lambda spo_{(k)}$ of an orthosymplectic character and an orthosymplectic character associated to a one-row tableau.

\section{Preliminaries}\label{sec:prelim}
A {\em partition} of a positive integer $N$ is a $k$-tuple of  positive integers
$\lambda=(\lambda_1,\ldots,\lambda_k)$ with $\lambda_1 \geq \lambda_2 \geq \cdots \geq \lambda_k$ and $\vert \lambda \vert =\sum_{i=1}^k \lambda_i=N$. The {\em Young diagram} of shape $\lambda$ contains $N$
boxes in $k$ left-justified rows with $\lambda_i$ boxes in the $i$th row.  For convenience, we will also denote the Young diagram of shape $\lambda$ by $\lambda$.  The {\em length} of the Young diagram of shape $\lambda$, denoted $\ell(\lambda)$, is equal to the number of rows in $\lambda$.  The {\em conjugate} of $\lambda$ is the partition
$\lambda^t=(\lambda_1^t,\lambda_2^t,\ldots,\lambda_s^t)$ where
$\lambda_i^t$ denotes the number of boxes in the $i$th column of the
Young diagram of shape $\lambda$.  A $\lambda$-tableau is obtained by filling the Young diagram of shape $\lambda$ with entries from a set
$\{1,2, \ldots, n\}$ of positive integers. A
$\lambda$-tableau is {\em semistandard} if the entries in the rows are weakly increasing from left to right and the entries in the
columns are strictly increasing from top to bottom.

Given partitions $\lambda$ and $\mu$, we write $\mu \subseteq \lambda$ if $\mu_i \leq \lambda_i$
for $i \geq 1$.
The {\em skew diagram} of shape $\lambda/\mu$ is formed by removing the Young diagram of shape $\mu$ from
the upper left-hand corner of the Young diagram of shape $\lambda$ and $\vert \lambda/\mu \vert =\vert \lambda \vert - \vert \mu \vert$.
A {\em skew tableau} of shape $\lambda/\mu$ is obtained by filling the skew diagram of shape $\lambda/\mu$ with
entries from a set $\{1,2,\ldots,n\}$ and is semistandard if the entries in the rows increase weakly from left to right and the entries in the columns increase
strictly from top to bottom.  A {\em horizontal strip} is a skew diagram $\lambda/\mu$ with no two boxes in the same column.

For a $\lambda$-tableau $T$, let $a_i(T)$ denote the number of entries equal to $i$ in $T$.  The {\em weight}
of $T$ is the monomial in the variables $X=\{x_{1},x_2,\ldots, x_n\}$ defined by $
\mathrm{wt}(T)=\prod_{i=1}^n x_i^{a_i(T)}. $  The {\em
 Schur function} corresponding to $\lambda$  is
$$s_{\lambda}(X)=\sum_{T} \mbox{wt}(T),$$ where the
sum runs over all semistandard $\lambda$-tableaux $T$ with entries in $\{1,2, \ldots,n\}$.
The {\em weight} of a skew tableau with entries from $\{1,\ldots,n\}$ and the
{\em skew Schur polynomial}, $s_{\lambda/\mu}(x_1,\ldots,x_n)$, are defined similarly.

Suppose that $\lambda$ is a partition in which the Young diagram of shape $\lambda$
has no more than $m$ rows.   A {\em semistandard symplectic $\lambda$-tableau} (see~\cite{king})  has entries in the set $B_0=\{1,\ov{1},2,\ov{2},\dots,m,\ov{m}\}$, is semistandard with respect to the ordering $1< \ov{1} < 2 < \ov{2} < \cdots < m < \ov{m}$
 and satisfies the property ({\em symplectic condition}) that the entries
in the $i$th row of $T$ are greater than or equal to $i$, for each
$1\leq i \leq m$.

Let $\overline{X}=\{x_1,x_1^{-1},\ldots,x_m,x_m^{-1}\}$. The weight
of a symplectic $\lambda$-tableau $T$   is $
\mathrm{wt}(T)=\prod_{i=1}^m x_i^{a_i(T)-a_{\ov{i}}(T)} $  where $a_i(T)$ (respectively $a_{\overline{i}}(T)$) is equal to the number of entries equal to $i$ (respectively $\overline{i}$)  in $T$.  The {\em symplectic Schur function} is
$$sp_{\lambda, 2m}(\overline{X})=\sum_{T} \mbox{wt}(T),$$ where the
sum runs over the semistandard symplectic $\lambda$-tableaux $T$ with entries in $\overline{X}$. 

\begin{exa} Let $\lambda=(3,2,1)$ and $m=4$.  The first $\lambda$-tableau below is semistandard symplectic, while the second is not, and $\mbox{wt}(T_1)=x_2x_3^{-2}x_4$.
$$T_1=\vy{1 & \ov{1} & 4 \cr 2 & \ov{3} \cr \ov{3} \cr}, \ \ \ T_2=\vy{1 & 1 & 4 \cr 2 & \ov{3} \cr \ov{2} \cr}$$ \end{exa}

Schur functions and symplectic Schur functions are  characters of  irreducible polynomial representations of the general linear groups $GL(n,\mathbb{C})$  and symplectic groups $Sp(2m,\mathbb{C})$, respectively.  We now turn our attention to the combinatorial properties of characters of orthosymplectic Lie superalgebras.  For background on the representation theory of Lie superalgebras, see \cite{cheng}.  Characters of orthosymplectic Lie superalgebras $\spo$ are hybrids of symplectic Schur functions and ordinary Schur functions and can be defined in terms of  $\spo$-tableaux  \cite{benkart}, which we now define.

Let $m$ and $n$ be  positive integers, $B_0=\{1, \ov1, 2, \ov2, \ldots, m, \ov{m}\}$, $B_1=\{1^\circ,2^\circ,\ldots,n^\circ\}$, $B=B_0 \cup B_1$ and order $B$ as follows:
\[ 1 < \ov1 < 2 < \ov2 < \cdots < m < \ov{m} < 1^\circ < 2^\circ < \cdots < n^\circ. \]

An {\em $spo(2m,n)$-tableau} $T$ of shape $\lambda$ (or an $spo$-tableau) is a filling of the Young diagram of shape $\lambda$ with entries from $B$ that satisfies the following conditions:

\begin{enumerate}
\item  the portion $S$ of $T$ that contains entries only from $B_0$ is semistandard symplectic;
\item the skew tableau formed by removing $S$ from $T$ is strictly increasing across the rows from left to right and weakly increasing down columns from top to bottom.
\end{enumerate}

Let $\overline{X}=\{x_1,x_1^{-1},\ldots,x_m,x_m^{-1}\}$, $Y=\{y_1,\ldots,y_n\}$ and $Z=\overline{X} \ \cup \ Y $ and let  $\lambda$ be a partition.
The {\em orthosymplectic character} corresponding to $\lambda$ is given by
\[ spo_{\lambda}(Z)=\sum_{\substack{\mu \subseteq \lambda \\ \ell(\mu) \leq m}} sp_{\mu}(\overline{X})s_{\lambda^t /\mu^t}(Y), \]
where $sp_\mu(\overline{X})$ is a symplectic Schur polynomial in the variables $\overline{X}$ and $s_{\lambda^t/ \mu^t}(Y)$ is a skew Schur polynomial in the variables $Y$ ~\cite[Theorem 4.24]{benkart}.

An orthosymplectic character can also be described in terms of $spo$-tableau.  To define the weight of an $spo$-tableau $T$, replace each $i \in \{1,2,\ldots,m\}$ that belongs to $T$ with $x_i$,
each $\overline{i} \in \{\ov{1},\ldots,\ov{m}\}$ in $T$ with $x_i^{-1}$ and each $i^\circ \in \{1^\c, \ldots,n^\c\}$ in $T$ with $y_i$ and
let  $\mbox{wt}(T)$ be the product of these variables.
Then $spo_{\lambda}(Z)=\sum_T \mbox{wt}(T)$,
where the sum is over the $spo$-tableaux of shape $\lambda$ with entries from $Z$ \cite[Theorem 5.1]{benkart} .

\begin{exa}
The following is an $spo(8,3)$-tableau of shape $\lambda=(3,3,2,2)$ and
 $\mbox{wt}(T)=x_1^{-1} x_4 y_1^2 y_2 y_3^3$.
 $$T=\vy{\ov{1} & 2 & 3^\c \cr \ov{2} & 1^\c & 3^\c \cr 4 & 1^\c \cr 2^\c & 3^\c \cr}$$
\end{exa}

\begin{exa}Let $\lambda=(1,1)$, $m=2$ and $n=1$.  Then $\overline{X}=\{x_1,x_1^{-1},x_2,x_2^{-1}\}$, $Y=\{y_1\}$  and
$$spo_\lambda(Z)=x_1x_2+x_1x_2^{-1}+x_1^{-1}x_2+x_1^{-1}x_2^{-1}+1+x_1y_1+x_1^{-1}y_1+x_2y_1+x_2^{-1}y_1+y_1^2.$$

\end{exa}

\section{Orthosymplectic insertion and the product of tableaux}\label{sec:insertion}

We first recall ordinary (Schensted-Knuth) {\em row insertion} for semistandard 
$\lambda$-tableaux (see \cite{fulton, sagan, schensted, stanley}).  Given a semistandard $\lambda$-tableau $T$ and a positive integer $x$, construct a new tableau, denoted $T \leftarrow x$, which has one more box than $T$ as follows:
if $x$ is greater than or equal to every entry in the first row, add $x$ in a new box at the end of the first row, giving a new tableau $T \leftarrow x$.  If this is not the case, replace the left-most entry in the first row that is strictly larger than $x$ with $x$ and bump the displaced entry into the second row by repeating the process with the displaced positive integer in row two.  Continue the process for successive rows until an entry lands in a new box at the end of a row or until an entry bumps into a new box at the bottom of the tableau, forming a new row.  The resulting tableau is $T \leftarrow x$ and the new box in $T \leftarrow x$ belongs to an {\em outer corner} of the Young diagram, which means that there is not a box directly below it or directly to its right.

Insertion for $spo$-tableaux has a similar flavour to the insertion algorithm for $(k,\ell)$-semistandard tableaux \cite{bereleregev, remmel} ($(k,\ell)$-semistandard tableaux arise in the description of characters of general linear Lie superalgebras).  In that setting, insertion consists of both ordinary row insertion and ordinary column insertion;   insertion for $spo$-tableaux  incorporates Berele's row insertion \cite{berele} for the symplectic part of a tableau (entries from $B_0$) and ordinary column insertion for the remaining entries from $B_1$. 
To define Berele's insertion algorithm, we first define a version of Schutzenberger's jeu de taquin algorithm for $spo$-tableau, which will also apply to symplectic tableaux.
 A {\em punctured $spo$-tableau} is obtained from an $spo$-tableau by removing the entry from one box.   An orthosymplectic version of jeu de taquin gives a way to move an empty box in the first column of a punctured $spo$-tableau through the tableau until the empty box lands in an outer corner.  A {\em forward slide} is achieved using the following moves:

$$\label{jdtmoves}\vy{ & a \cr b \cr } \rightarrow
\left\{\begin{array}{ll}
\vy{a &  \cr  b \cr} & \text{if $a < b$ or $a=b$ and $a,b \in B_1$} \\[5mm]
\vy{b & a \cr  \cr} & \text{if $b<a $ or $a=b$ and $a,b \in B_0$.}
\end{array}\right.\ \ \ 
$$
 
 When a forward slide is performed on an empty box in a punctured $spo$-tableau, the result is a punctured $spo$-tableau.  Furthermore, given a punctured $spo$-tableau, where the empty box occurs in the first column, forward slides can be continued until the empty box lands in an outer corner \cite{benkart}.  The procedure is also reversible and we will refer to the slides that reverse the procedure as {\em reverse slides}.  
 
To describe Berele row insertion, consider a semistandard symplectic tableau $T$ and $x \in  B_0$.  If  ordinary row insertion
of $x$ into $T$ does not introduce a violation of the symplectic condition, then the new tableau produced by ordinary row insertion is the result of Berele row insertion.  If  a symplectic violation is introduced, it must be the case that, for some $i$,  an $\overline{i}$ was bumped out of row $i$ and into row $i+1$ by an $i$.  Find the least $i$ for which this occurs and, at this stage, instead of replacing the $\overline{i}$ with an $i$, remove the $\overline{i}$ from the tableau, leaving a hole in the box that contained it.   Apply jeu de taquin forward slides to the empty box until it lands in an outer corner, giving a semistandard symplectic tableau that has one less box than $T$.  This will be referred to as a cancellation and this new tableau is the result of the Berele insertion of $x$ into $T$. 

\bigskip

\begin{exa}  In the following example, the entry $\overline{1}$ is inserted into the symplectic tableau $T$.
$$T=\vy{1 & \overline{1} & {\bf 2} & 3 \cr \overline{2} & 3 \cr \overline{3} & 4  \cr}\shiftbar{.45}{1} \ \ \ \vy{1 & \overline{1} & \overline{1} & 3 \cr & 3 \cr \overline{3} & 4 \cr}; \ \ \ \vy{1 & \overline{1} & \overline{1} & 3 \cr 3 & 4 \cr \overline{3} \cr}$$
\end{exa}

We now define insertion of a letter $x \in B_0 \cup B_1$ into an $spo$-tableau $T$, as in \cite[\S 5.3]{benkart}.  We will refer to the procedure as $spo$-insertion and if $x \in B_0$, insertion of $x$ into $T$ will be denoted $T \leftarrow x$, while if $x \in B_1$, insertion of $x$ into $T$ will be denoted $x \rightarrow T$.  This is consistent with notation used in \cite{schensted}, since entries in $B_0$ will be inserted into rows, while entries in $B_1$ will be inserted into columns.

We first describe how to insert an entry $x$ into a given row or column. If $x \in B_1$ is greater than or equal to every entry in a given column, place $x$ in a new box at the end of that column.  If not, replace the least entry $y$ in the column that satisfies $y>x$ with $x$, which displaces $y$. 
For entries $x$ belonging to $B_0$, insert $x$ into a given row by placing it  in a new box at the end of the row if $x$ is larger than every entry in the row.  Otherwise,   replace the least entry $y$ in the row that satisfies $y>x$ with $x$, displacing $y$, provided that doing so does not cause $x$ to replace an entry $\overline{x}$ in row $x$.  If this is the case, remove  $\overline{x}$ from row $x$, leaving an empty box.

In general, to insert $x$ into an $spo$-tableau $T$, begin by inserting $x$ into the first row of $T$ if $x\in B_0$ and insert $x$ into the first column of $T$ if $x \in B_1$, as described above.  If doing so displaces an entry $y\in B_0$, insert $y$ into the subsequent row using the same procedure, while if an entry $y\in B_1$ is displaced, insert $y$ into the subsequent column.  On the other hand, if  an empty box arises, move the empty box to an outer corner of the tableau using jeu de taquin slides.   Repeat the procedure until an entry lands in a new box at the end of a row or until an empty  box moves to an outer corner.   The process results in an $spo$-tableau $T \leftarrow x$, if $x \in B_0$, and $x \rightarrow T$, if $x \in B_1$.  When a cancellation occurs, the new tableau $T \leftarrow x$ has one less box than $T$.  If a cancellation does not occur, $T \leftarrow x$ (respectively $x \rightarrow T$) has one more box than $T$ and the new  box belongs to an outer corner.
\begin{exa}

The following example illustrates the steps in the $spo$-insertion  process, where $\overline{1}$ is inserted into the $spo$-tableau $T$.

$$T=\vy{ 1 & {\bf 2} & \overline{2} \cr 2 & 3 & 4 \cr \overline{3} & 1^\circ \cr 1^\circ & 2^\circ \cr} \shiftbar{.62}{1}\ \ \ \ \ \vy{1 & \overline{1} & \overline{2} \cr 2 & {\bf 3} & 4 \cr \overline{3} & 1^\circ \cr 1^\circ & 2^\circ \cr} \shiftnobar{.2}{2} \ \ \ \ \  
\vy{1 & \overline{1} & \overline{2} \cr 2 & 2 & 4 \cr {\bf \overline{3}} & 1^\circ \cr 1^\circ & 2^\circ \cr} \shiftnobar{-.3}{3} \ \ \ \ \ 
\vy{1 & \overline{1} & \overline{2} \cr 2 & 2 & 4 \cr & 1^\circ \cr 1^\circ & 2^\circ \cr};  \ \ \ \ \ \vy{1 & \overline{1} & \overline{2} \cr 2 & 2 & 4 \cr 1^\circ & 2^\circ \cr 1^\circ \cr}$$

\end{exa}

Let $x \in B_0$,  suppose that $x$ does not cause a cancellation when  inserted into an $spo$-tableau $T$ and suppose that $T \leftarrow x$ has shape $\nu$.  Define the {\em sp-bumping route} of $x$ to be the collection of boxes from the Young diagram of shape $\nu$ that consist of those boxes from which an element from $B_0$ was bumped from a row, together with the last box in which an element from $B_0$ lands.  This last box will be referred to as the final box in the bumping route for $x$.  As in \cite{fulton}, we say that an $sp$-bumping route $R_1$ sits strictly right (or weakly right) of an $sp$-bumping route $R_2$ if for every row in the Young diagram that contains a box of $R_1$, $R_2$ contains a box that is right of the box in $R_1$.  
We define $R_1$ to be above (or weakly above) $R_2$ similarly.

\begin{exa}  The $sp$-bumping route for the row-insertion of $3$ into the given tableau $T$ is marked in bold.
$$\vy{\ov{1} & \ov{3} & 5^\circ \cr 4 & 4^\circ \cr 3^\circ & 4^\circ \cr} \leftarrow 3 =\vy{\ov{1} & {\bf 3} & 4^\circ & 5^\circ \cr {\bf \ov{3}} & 3^\circ \cr {\bf 4} & 4^\circ \cr}$$

\end{exa}

In what follows, we will associate each term in an orthosymplectic Schur function $spo_\lambda$ with an $spo$-tableau $T$ of shape $\lambda$,  and each term in an orthosymplectic Schur function $spo_{(k)}$, where $k$ is a positive integer, with a one-row $spo$-tableau $U$ with $k$ boxes.  The result of multiplying a term from $spo_\lambda$ with a term from $spo_{(k)}$ will correspond to a certain $\nu$-tableau, which has shape corresponding to the {\em product} $T\cdot U$, which we now define. 

Given an $spo$-tableau $T$ and a one-row $spo$-tableau $U=\vy{a_1 & a_2 & \ldots & a_r & b_1 & b_2 & \ldots & b_t \cr}$, where $a_i \in B_0$ with $a_1 \leq a_2 \leq \cdots \leq a_r$ and  $b_i \in B_1$ with $b_1 <b_2< \cdots < b_t$, define $$T \cdot U = (b_1 \rightarrow \cdots \rightarrow b_t \rightarrow T) \leftarrow a_1 \leftarrow a_2 \leftarrow \cdots \leftarrow a_r.$$
The tableau resulting from each stage of the $spo$-insertion process is an $spo$-tableau \cite[Lemma 5.4]{benkart} so  $T \cdot U$ is an $spo$-tableau. 

\begin{exa} The tableaux involved in each step of the construction of $T \cdot U$, where $U=\vy{1 & 3 & 5 & 2^\circ & 5^\circ \cr}$ are given below.  
$$T=\vy{\overline{1} & 2 & 5 \cr 3 & \overline{4} & 1^\circ \cr 4 & 5^\circ \cr 2^\circ \cr}, \ \ 
\vy{\ov{1} & 2 & 5 \cr 3 & \ov{4} & 1^\circ \cr 4 & 5^\circ \cr 2^\circ \cr 5^\circ \cr}, \ \ 
\vy{\ov{1} & 2 & 5 \cr 3 & \ov{4} & 1^\circ \cr 4 & 5^\circ \cr 2^\circ & 5^\circ \cr 2^\circ \cr }, \ \ 
\vy{2 & \ov{4} & 5 \cr 3 & 1^\circ \cr 4 & 5^\circ \cr 2^\circ & 5^\circ  \cr 2^\circ \cr}, \ \ \vy{2 & 3 & 5 \cr 3 & \ov{4} & 1^\circ \cr 4 & 5^\circ \cr 2^\circ  & 5^\circ \cr 2^\circ \cr}, \ \ \vy{2 & 3 & 5 & 5 \cr 3 & \ov{4} & 1^\circ \cr 4 & 5^\circ \cr 2^\circ & 5^\circ \cr 2^\circ \cr} = T \cdot U$$

\end{exa}

\section{An orthosymplectic Pieri rule}\label{sec:main}

Our main result is a Pieri rule for orthosymplectic Schur functions.
\begin{thm}\label{mainthm} Let $\lambda$ be a partition and $k$ be a positive integer.  Then $$spo_\lambda \cdot spo_{(k)}=
\sum_\nu \alpha_\nu spo_\nu, \mbox{ where}$$
$$\alpha_\nu= \vert \{ \mu \mid \mu \subseteq \lambda, \mu \subseteq \nu, \lambda/\mu,  \nu/\mu \mbox{ are horizontal strips and } \vert\nu/\mu \vert+\vert \lambda/\mu \vert=k\}\vert.$$  \end{thm}

Theorem \ref{mainthm} establishes that orthosymplectic  Schur functions obey the same Pieri rule as Sundaram's Pieri rule for symplectic Schur functions \cite[Theorem 4.1]{sundaram}.  Since the Jacobi-Trudi identity for orthosymplectic Schur functions matches that for symplectic Schur functions (see \cite{benkart} and \cite[Corollary 4.3]{stokkevisentin}), the structure coefficients that appear in the expansion of the product of two orthosymplectic Schur functions are the same as those that appear in the expansion of the product of the corresponding symplectic Schur functions. This situation is similar to that which occurs for characters of general linear Lie superalgebras \cite{bereleregev, remmel}.

\begin{cor}  Suppose that $sp_\lambda \cdot sp_\mu=\sum \alpha_\nu sp_\nu$.  Then $spo_\lambda \cdot spo_\mu=\sum \alpha_\nu spo_\nu$.
\end{cor}

We now set about proving several lemmas from which Theorem \ref{mainthm} follows.  Lemmas \ref{flemma} -- \ref{lastlem} allow us to describe the shape $\nu$ of a tableau that is the product of an $spo$-tableau and a one-row $spo$-tableau in Theorem \ref{shapethm}.

Lemmas \ref{flemma} and \ref{splem} follow from known results about ordinary insertion into semistandard tableaux.  Proofs can be found in, for instance, \cite[\S 1.1]{fulton}.

\begin{lem}\label{flemma} Suppose that $x_1,x_2 \in B_1=\{1^\circ,\ldots, n^\circ\}$, with $x_1>x_2$ and that the box added to the Young diagram of $T$ in $x_1 \rightarrow T$ is $S_1$, and the box added to the shape of $x_1 \rightarrow T$ in $x_2 \rightarrow (x_1 \rightarrow T)$ is $S_2$.  Then $S_2$ is weakly above and strictly right of $S_1$.

\end{lem}

\begin{cor}Suppose that $x_1,\ldots,x_r \in B_1$, with $x_1>x_2>\cdots>x_r$, that $T$ is an $spo$-tableau of shape $\lambda$, and  that $S=x_r \rightarrow \cdots \rightarrow x_1 \rightarrow T$ has shape $\mu$.  Then $\mu/\lambda$ is a horizontal strip.\end{cor}

\begin{lem}\label{splem} Suppose that $x_1, x_2 \in B_0=\{1,\ov{1},\ldots,m,\ov{m}\}$ with $x_1 \leq x_2$ and that  neither insertion of $x_1$ into an $spo$-tableau $T$ nor insertion of  $x_2$ into the resulting tableau $T \leftarrow x_1$  introduce cancellations. Let $R_1$ be the $sp$-bumping route for $x_1$ and $R_2$ the $sp$-bumping route for $x_2$.  Then the final box in $R_2$ is  weakly above and strictly right of the final box in $R_1$.
\end{lem}

\begin{lem}\label{mylemma} Suppose that $x \in B_1$ is inserted into an $spo$-tableau $T$ giving a new box $S_1$ in $x \rightarrow T$, and  insertion of $y \in B_0$ into $x \rightarrow T$ does not cause a cancellation and gives a new box $S_2$ in $(x \rightarrow T) \leftarrow y$. Then $S_2$ is weakly above and strictly right of $S_1$.  \end{lem}

\begin{proof}  Let $U=x \rightarrow T$ and suppose that $S_1$ belongs to column $j$ of  $U$ and contains $x_j \in B_1$.  Then there are entries $x_1,\ldots, x_{j-1}$, where $x_i$ belongs to column $i<j$ of $U$ and $x_i \in B_1$, such that $x_1 < x_2 < \cdots < x_j$.  If the $sp$-bumping route of $y$ ends in a column strictly right of column $j$, the result follows so suppose that the $sp$-bumping route of $y$ ends in column $k$, where $k\leq j$.  Then,  since each of columns $1, \ldots, k$ contain entries from $B_1$, the entry in the final box of the $sp$-bumping route of $y$ upon insertion into $U$ bumps an entry $a \in B_1$,  with $a \leq x_k$, out of column $k$ and into column $k+1$, which then bumps an entry from $B_1$ that is less than or equal to $x_{k+1}$ into column $k+2$, et cetera.  This continues until an entry from $B_1$ that is less than or equal to $x_j$ is bumped out of column $j$ and into column $j+1$, which necessarily bumps into a box in the same row or in a row above the row containing $S_1$.  It follows that $S_2$ belongs to column strictly right of column $j$ and  is weakly above the box $S_1$. \end{proof}

\begin{lem}\label{prevlem} Suppose that $x_1, x_2 \in B_0$ with $x_1 \leq x_2$ and that  neither insertion of $x_1$ into an $spo$-tableau $T$ nor insertion of  $x_2$ into the resulting tableau $T \leftarrow x_1$  introduce cancellations.  Let $S_1$ denote the new box in $U=T \leftarrow x_1$ that results from the $x_1$ insertion and $S_2$ the new box in $U \leftarrow x_2$ that results from the $x_2$ insertion.   Then
$S_2$ is  strictly right of and weakly above $S_1$. \end{lem}

\begin{proof}
Suppose that neither insertion of $x_1$ into $T$ nor insertion of $x_2$ into $U$ displace entries from $B_1$. Then  $S_1$ is the final box in the $sp$-bumping route of $x_1$  and $S_2$ is the final box in the $sp$-bumping route of $x_2$ so the result follows from Lemma \ref{splem}.

Suppose that insertion of $x_1$ does not displace an entry from $B_1$, but insertion of $x_2$ into $U$ does displace an entry from $B_1$.  Then $S_1$ is the final box in the $sp$-bumping route of $x_1$ and the final box in the $sp$-bumping route of $x_2$ is strictly right of and weakly above $S_1$  by Lemma \ref{splem}.  But $x_2$ insertion bumps an entry from $B_1$ out of this final box in the $sp$-bumping route into a box in the next column that is strictly right of and weakly above it and this continues until the new box $S_2$ is added to $U$, so $S_2$  is strictly right of and weakly above $S_1$.

  If insertion of $x_1$ into $T$ displaces an entry from $B_1$, but insertion of $x_2$ into $U$ does not, then $S_2$ must belong to a column strictly right of that which contains $S_1$.  Indeed, since insertion of $x_1$ displaces an entry from $B_1$,   each outer corner in columns $1, \ldots, j-1$ of $U$ must contain an entry from $B_1$.  Since $S_2$ is an outer corner and insertion of $x_2$ does not displace an entry from $B_1$, $S_2$ must belong to a column to the right of column $j$.  Since $S_2$ contains an entry from $B_0$, while $S_1$ does not, $S_2$ must belong to a row above $S_1$.

Suppose that both insertions displace entries from $B_1$ and suppose that the final box in the $sp$-bumping route for $x_2$ belongs to column $k$.  By Lemma \ref{splem}, the final box for the $sp$-bumping route of $x_1$ is strictly left of column $k$.  If the box $S_1$ that is added to $T$ by $x_1$ insertion belongs to column $j\leq k$, then, since $x_2$ insertion bumps an entry from column $k$ into column $k+1$ and $S_1$ and $S_2$ are both at the end of columns, $S_1$ belongs to a row below $S_2$.   Otherwise, suppose that insertion of $x_1$ into $T$ adds the box $S_1$ to a column $j > k$.  Then $x_1$ insertion bumps an entry $s_1 \in B_1$ from column $k$ and into column $k+1$, replacing $s_1$ with an entry $s_2<s_1$, where $s_2 \in B_1$.  Since the final box for the $sp$-bumping route of $x_2$ is in column $k$, $x_2$ insertion displaces the smallest entry $s_3 \in B_1$ in column $k$ and bumps it into column $k+1$.  But $s_3 \leq s_2<s_1$ and the final box produced by $x_1$ insertion is determined by insertion of $s_1$ into column $k+1$ and the remaining columns of $T$, while the final box produced by $s_2$ insertion is determined by insertion $s_3$ into column $k+1$ and the remaining columns of $T \leftarrow x_1$.  The result follows from  Lemma \ref{flemma}.  \end{proof}

\begin{lem}\label{sundlem1}\cite[Lemma 3.2]{sundaram}
Suppose that $T$ is an $spo$-tableau and that $x_1, x_2 \in B_0$ with $x_1 \leq x_2$.  If insertion of $x_2$ into $T \leftarrow x_1$ causes a cancellation, then insertion of $x_1$ into $T$ also causes a cancellation. 

\end{lem}

\begin{cor}  Let $T$ be an $spo$-tableau and $U$ a one-row $spo$-tableau. Entries in $U$ that cause cancellations when $T \cdot U$ is constructed occur in an initial strip at the beginning of $U$.

\end{cor}

\begin{lem} \label{rowonelem} Suppose that $x_1 \in B_1$ and that insertion of $x_1$ into $T$ adds a new box to row one of the Young diagram of $T$.  Then if $x_2 \in B_0$ and insertion of $x_2$  into  $U=x_1 \rightarrow T$ causes a cancellation, the tableau $U \leftarrow x_2$ and $U$ have the same number of boxes in the first row.  In other words, the cancellation does not result in the removal of a box from the first row.
\end{lem}
\begin{proof}
The added box in  row one of $U$ contains an entry from $B_1$.  Suppose that column $k$ is the leftmost column in which an entry from $B_1$ bumps into the first row during the $x_1$ insertion process and let $a \in B_1$ denote the entry that belongs to column $k$ in the first row of $U$.  Then column $k-1$ of $U$ contains an entry $b \in B_1$, where $b<a$.  
Suppose that insertion of $x_2$ into $U$ causes a cancellation and that moving the resulting empty box through the tableau via jeu de taquin results in the removal of a box from row one of $U$.  Then the cancellation must occur in the first row of $T$ and the resulting empty box moves to the end of row one via jeu de taquin.  This then moves the entry $a$ in row one  above the entry $b$, which is a contradiction, since $b<a$.  \end{proof}

\begin{lem}\label{sundlem4}\cite[Lemma 3.3]{sundaram}  Suppose that $x_1, x_2 \in B_0$ with $x_1 \leq x_2$ and that $T$ is a symplectic tableau.  Suppose that $x_1$ causes a cancellation when inserted into $T$ via the Berele insertion algorithm  and that $x_2$ causes a cancellation when inserted into the tableau $T \leftarrow x_1$.  Then the box removed from $T \leftarrow x_1$ by insertion of $x_2$  sits weakly below and strictly left of the  box removed from $T$ by insertion of $x_1$ into $T$.\end{lem}

\begin{lem} Suppose that $x_1, x_2 \in B_0$ with $x_1 \leq x_2$ and suppose that $x_1$ causes a cancellation when inserted into an $spo$-tableau $T$ and that $x_2$ causes a cancellation when inserted into $U=T \leftarrow x_1$.  Then the box removed from $U$ by insertion of $x_2$  belongs to a different column than the box removed from $T$ by insertion of $x_1$.

\end{lem}

\begin{proof}

The empty box caused by cancellation when $x_1$ is inserted into $T$ slides from the first column through the $sp$-portion of the tableau and into the remaining portion of the tableau via jeu de taquin.    Let $S_1$ denote the last box in the jeu de taquin path of the empty box such that an entry from  $B_0$ slides out of $S_1$.  Then  $S_1$ either belongs to an outer corner of $T$, in which case the insertion procedure is complete or an entry from $B_1$ slides into $S_1$.  If $S_2$ is the last box in the jeu de taquin path of the empty box created by $x_2$ such that an entry from  $B_0$ slides out of $S_2$, then $S_2$ is weakly below and strictly left of the box $S_1$ by Lemma \ref{sundlem4}  so the two boxes belong to different columns of $T$.  If one or both boxes are the boxes removed through orthosymplectic insertion via $x_1$ and $x_2$ insertion respectively, the result follows.    If the jeu de taquin process continues through the remaining portion of the tableau by sliding entries from $B_1$ into the boxes $S_1$ and $S_2$ and continuing, the two vacant boxes migrate to outer corners in different columns.  This follows from \cite{leeuwen}, where the author proves jeu de taquin results through tableau switching.  \end{proof}

\begin{lem}\label{lastlem}
Suppose that $y \in B_1$ is inserted into an $spo$-tableau $T$ and that $x_1,\ldots,x_r \in B_0$ with $x_1\leq \ldots \leq x_r$ are subsequently inserted into the remaining tableaux producing $(y \rightarrow T) \leftarrow x_1 \leftarrow \cdots  \leftarrow x_{r-1} \leftarrow x_r$ and suppose that $x_{r-1}$ causes a cancellation while $x_r$ does not.  Then insertion of $x_r$ into
$(y \rightarrow T) \leftarrow x_1 \leftarrow \cdots \leftarrow x_{r-1}$ produces a new box in a column that is strictly right of the box added by insertion of  $y$ into $T$.

\end{lem}

\begin{proof}
Suppose that insertion of $y$ augments the tableau $T$ by a box in column $j$.  Then each of columns $1$ through $j$ in $y \rightarrow T$ contain an entry from $B_1$. Furthermore, if $a_1$ is the smallest entry from $B_1$ that belongs to the first column of $y \rightarrow T$, then entries $a_2, \ldots, a_j \in B_1$ can be chosen, where $a_i$ in column $i$, with $a_1 <a_2< \ldots < a_j$.  Take $a_2$ in   the second column of $y \rightarrow T$ to be minimal with the property that  $a_1 <a_2$ and, generally, take $a_{i+1}$ to be minimal in column $i+1$ with the property that $a_i <a_{i+1}$, for $2 \leq i \leq j$.  Then $a_{i+1}$ belongs to a row weakly above $a_i$ in $y \rightarrow T$.  When a cancellation occurs and an empty box in column one slides through the tableau, none of the entries $a_i$ move out of their respective columns since, otherwise, an entry $a_i$ would slide into column $i-1$ above $a_{i-1}$, which is impossible.  Thus, if $S=(y \rightarrow T) \leftarrow x_1\cdots \leftarrow x_{r-1}$ then $a_i$ also belongs to column $i$ of $S$ for $1 \leq i \leq j$.

If the final box in the $sp$-bumping route for  $x_r$ insertion into $S$ belongs to a column strictly right of column $j$, the result follows so suppose that the final box in the $sp$-bumping route for  $x_r$  is in column $i$, which is weakly left of column $j$. Then an entry from $B_0$ displaces an entry from  $B_1$ that is in this final box in column $i$ and moves it into column $i+1$.   Since the displaced entry is the smallest entry in column $i$ that belongs to  $B_1$, it is less than or equal to $a_i$ and cannot land at the bottom of column $i+1$ since there is at least one entry in column $i+1$ - namely $a_{i+1}$ - that is larger than it.  Since there are entries $a_{i+1}<a_{i+2}<\cdots < a_j$ in each of the subsequent columns, which belong to $B_1$, an entry in $B_1$ cannot get appended to the bottom of any column prior to, and including, column $j$.  Thus the new box belongs to a column strictly right of column $j$.  
\end{proof}

The above lemmas show that the shape $\nu$ of $T \cdot U$ takes a particular form.

\begin{thm}\label{shapethm} Let $T$ be an $spo$-tableau of shape $\lambda$ and $U$ a semistandard one-row $(k)$-tableau containing entries $x_s \leq  \cdots \leq x_1 \in B_0$ and $a_1 <\cdots < a_{m_1} \in B_1$.  Then $T \cdot U$ is formed in three stages, producing intermediate tableaux of shapes $\mu^\circ$ and $\mu^-$:  an addition stage that yields a tableau  $T_{\mu_0}=a_1 \rightarrow \cdots \rightarrow a_{m_1} \rightarrow T$, then a cancellation stage, which produces  $T_{\mu^-}=T_{\mu_0}\leftarrow  x_s \leftarrow \cdots \leftarrow x_{m_2+1}$, for some $0 \leq m_2 \leq s$, followed by another addition stage, which gives $T_{\nu}=T_{\mu^-} \leftarrow x_{m_2} \leftarrow \cdots \leftarrow x_1$.  Furthermore, $T \cdot U$ has shape $\nu$, where $\nu$  is the end shape in a sequence of Young diagrams
$(\lambda,\mu^\circ,\mu^-,
\nu)$ that satisfies the following properties:
\begin{enumerate} 
\item $\mu^\circ/\lambda$, $\mu^\circ/\mu^-$ and $\nu/\mu^-$ are horizontal strips;
\item  $\vert \mu^\circ/\lambda \vert =m_1$, for some $0 \leq m_1 \leq k$;
\item $\vert \mu^\circ/\mu^-\vert$=$\ell$, for some $0 \leq \ell \leq k-m_1$, and if the first rows of $\lambda$ and $\mu^0$ do not contain the same number of boxes then the first rows of $\mu^\circ$ and $\nu$ do contain the same number of boxes;
\item $\vert \nu / \mu^-\vert=k-m_1-\ell$ and the leftmost box in $\nu/\mu^-$ belongs to a column that is strictly right of the rightmost box in $\mu^0 / \lambda$.

\end{enumerate}
 \end{thm}

We will call a sequence $(\lambda, \mu^\circ,\mu^-,\nu)$, where the Young diagrams satisfy the properties in Theorem \ref{shapethm}, an {\em $spo$-bumping sequence}.  A triple $(\lambda,\mu,\nu)$, where $\mu \subseteq \lambda$ and $\mu \subseteq \nu$ will be called an {\em $sp$-bumping sequence} if $\lambda / \mu$ and $\nu / \mu$ are horizontal strips.  
The triple $(\lambda,\mu,\nu)$, where $\vert \lambda / \mu \vert =\ell$ and $\vert \nu / \mu \vert=k-\ell$, for $k \geq 0$ and $0 \leq \ell \leq k$, is an $sp$-bumping sequence if and only if the symplectic Schur function $sp_\nu$ appears in the expansion of the product of $sp_\lambda sp_{(k)}$ \cite{sundaram}.  The shape $\mu^\circ$ in the $spo$-sequence $(\lambda,\mu^\circ,\mu^-,
\nu)$ described in Theorem \ref{shapethm} is that obtained by column-insertion of the entries in $U$ that belong to $B_1$.  The triple $(\mu^\circ, \mu^-, \nu)$ is an $sp$-bumping sequence so for symplectic tableaux, the result follows from \cite{sundaram}.  

If $S=T \cdot U$ and the $spo$-bumping sequence corresponding to the product is $(\lambda,\mu^\circ,\mu^-,\nu)$, then the $spo$-insertion algorithm can be reversed to decompose $S$ as a product.  Working from the sequence, $m_1=\vert \mu^\circ/ \lambda \vert $, $\ell=\vert \mu^\circ / \mu^- \vert$, $m_2=\vert \nu /\mu^- \vert$ and $k=m_1+m_2+ \ell$.  If the entry $x$ in the rightmost box of $\nu/\mu^-$ belongs to $B_0$, then if $x$ is in the first row of $S$, it bumps out of $S$.  Otherwise, bump $x\in B_0$ into the row above it by replacing the largest entry in that row that is less than it.  Then bump the displaced entry into the row above it and continue until an entry from $B_0$ bumps out of the first row of the tableau.  If $x \in B_1$, bump $x$ into the column immediately left of it by replacing the largest entry less than it in the column with $x$.  If the displaced entry belongs to $B_1$, bump the displaced entry into the column immediately left of it in a similar way.  On the other hand, if the displaced entry belongs to $B_0$, bump it into the row above it as described above. Continue bumping displaced entries in this way until an entry from  $B_0$ bumps out of row one.   The procedure can be continued with the entries in the remaining boxes of $\nu/ \mu^-$ until $m_2$ entries $x_i$ have been bumped out of the tableau in succession with $x_1 \geq x_2 \geq \cdots \geq x_{m_2}$, since $\nu/\mu^-$ is a horizontal strip.  This yields a tableau $T_{\mu^-}$ of shape $\mu^-$.

The next portion of the reverse algorithm, which is part of the argument used to prove the symplectic Pieri rule \cite{sundaram},
 produces the symplectic part of the one-row tableau  $U$ that causes cancellations during the insertion procedure.  Add empty boxes to $T_{\mu^-}$ to produce a punctured tableau of shape $\mu^\circ$ with $\ell$ empty boxes.  Move the leftmost empty box through the tableau using reverse jeu de taquin slides until it lands in the highest box possible in column one without causing a symplectic violation.  Then place an $\ov{i}$ in the empty box, where $i$ is equal to the row number in which the empty box sits.  Then bump an $i$ into the row above the $\ov{i}$ by replacing the largest entry less than $i$ in that row with $\ov{i}$ and continue bumping up through the rows until an entry $x_{m_2+1}$ is bumped out of row one.  Continue the procedure with the empty boxes until $\ell$ entries $x_{m_2+1},\ldots,x_s$ from $B_0$ have been bumped out of the tableau in succession with $x_{m_2+1} \geq \cdots \geq x_{s}$.  This yields a tableau $T_{\mu^\circ}$ of shape $\mu^\circ$.

Finally, consider the rightmost box in  $\mu^\circ$ that does not belong to  $\lambda$; the  entry in this box of $T_{\mu^\circ}$ necessarily belongs to $B_1$.  If the entry belongs to the first column of $T_{\mu^\circ}$, it bumps out of $T_{\mu^\circ}$.  Otherwise, bump the entry  into the column immediately left of it by replacing the largest entry less than it and continue until an entry $a_1 \in B_1$ bumps out of the first column.  Repeat the process for each of the entries in the boxes that belong to $\mu^\circ$ but not to $\lambda$ until entries $a_1 <a_2< \cdots < a_{m_1}$ have been bumped out of the first column.  The tableau remaining is $T$, which has shape $\lambda$, and 
\newdimen\Squaresize \Squaresize=18pt $U=\vy{\ x_{s}  & \cdot \cdot & x_1 & a_1 & \cdot \cdot &\  a_{m_1}\ \cr}$.  

\begin{exa}

\end{exa} Suppose that $\lambda=(3,3,3,2)$, $\mu^\circ=(4,3,3,2,2)$, $\mu^-=(4,3,2,2,2)$ and $\nu=(6,3,2,2,2)$.  Then $m_1=3, \ell=1$ and $m_2=2$.  The following details the procedure for decomposing the $\nu$-tableau $S$ as a product of a $\lambda$-tableau and a one-row $(6)$-tableau.  The entries bumped out of the tableau are shown above the arrows at each stage.

\newdimen\Squaresize \Squaresize=14pt 
$$S=\vy{2 & 3 & 5 & 5 & 5^\circ & 6^\circ \cr 3 & \ov{4} & 1^\circ \cr 4 & 4^\circ \cr 2^\circ & 5^\circ \cr 2^\circ & 5^\circ \cr}  \overset{5}\longrightarrow \vy{2 & 3 & 5 & 5^\circ & 6^\circ \cr 3 & \ov{4} & 1^\circ \cr 4 & 4^\circ \cr 2^\circ & 5^\circ \cr 2^\circ & 5^\circ \cr}  \overset{3}\longrightarrow \vy{2 & \ov{4} & 5 & 6^\circ \cr 3 & 1^\circ & 5^\circ \cr 4 & 4^\circ \cr 2^\circ & 5^\circ \cr 2^\circ & 5^\circ \cr}  \longrightarrow \vy{2 & \ov{4} & 5 & 6^\circ \cr 3 & 1^\circ & 5^\circ \cr 4 & 4^\circ & \cr 2^\circ & 5^\circ \cr 2^\circ & 5^\circ \cr} $$
$$ \longrightarrow \vy{ & 2 & 5 & 6^\circ \cr 3 & \ov{4} & 1^\circ \cr 4 & 4^\circ & 5^\circ \cr 2^\circ & 5^\circ \cr 2^\circ & 5^\circ \cr} \overset{1}\longrightarrow \vy{\ov{1} & 2 & 5 & 6^\circ \cr 3 & \ov{4} & 1^\circ \cr 4 & 4^\circ & 5^\circ \cr 2^\circ & 5^\circ \cr 2^\circ & 5^\circ \cr} \overset{2^\circ}\longrightarrow \vy{\ov{1} & 2 & 5 \cr 3 & \ov{4} & 1^\circ \cr 4 & 5^\circ & 6^\circ \cr 2^\circ & 5^\circ \cr 4^\circ & 5^\circ \cr} \overset{4^\circ}\longrightarrow \vy{\ov{1} & 2 & 5 \cr 3 & \ov{4} & 1^\circ \cr 4 & 5^\circ & 6^\circ \cr 2^\circ & 5^\circ \cr 5^\circ \cr} \overset{5^\circ}\longrightarrow \vy{\ov{1} & 2 & 5 \cr 3 & \ov{4} & 1^\circ \cr 4 & 5^\circ & 6^\circ \cr 2^\circ &  5^\circ \cr}$$
$$S=\vy{\ov{1} & 2 & 5 \cr 3 & \ov{4} & 1^\circ \cr 4 & 5^\circ & 6^\circ \cr 2^\circ & 5^\circ \cr} \cdot \vy{1 & 3 & 5 & 2^\circ & 4^\circ  & 5^\circ \cr}$$

\bigskip

The following Lemma is the key to proving the equivalence of the symplectic and orthosymplectic Pieri rules.

\begin{lem}\label{bijectionlem} Let $\ell$ and $k$ be nonnegative integers with $\ell \leq k$.  There is a bijection between the set of pairs $((\lambda,\mu,\nu),m_1)$, consisting of $sp$-bumping sequences $(\lambda,\mu,\nu)$ where $\vert \lambda/ \mu \vert =\ell$ and $\vert\nu / \mu \vert = k-\ell$, and $spo$-bumping sequences $(\lambda,\mu^\circ,\mu^-,\nu)$, where $\vert \mu^\circ / \lambda \vert =m_1$, $\vert \mu^\circ / \mu^- \vert = \ell$, and $\vert \nu / \mu^- \vert =m_2=k-\ell-m_1$. \end{lem}

\begin{proof}
Given an $spo$-bumping sequence $(\lambda,\mu^\circ,\mu^-,\nu)$,  if $\mu^\circ=\mu^-$, then  $\ell=0$ and the corresponding $sp$-bumping sequence is $(\lambda,\mu,\nu)$, where $\mu=\mu^\circ$.  Otherwise, $\mu^-\subset \mu^\circ$ and we can form an $sp$-bumping sequence $(\lambda,\mu,\nu)$, where $\vert \lambda/ \mu \vert=\ell$, by first placing dots in all boxes in $\mu^\circ$ that belong to $\mu^\circ$ but not to $\mu^-$.    If a box  in the first row of $\mu^\circ$ contains a dot, there must also be a box in that position in $\lambda$ by Lemma \ref{rowonelem} and all  boxes containing dots are in different columns. Working from top to bottom and right to left, if a  box containing a dot belongs to $\lambda$, place a dot in the box in $\lambda$.  If the dotted box does not belong to $\lambda$, and belongs to row $i$ in $\mu^\circ$, then place a dot in  the rightmost box of row $i-1$ of $\lambda$ that does not already contain a dot.   In this way, a unique box in $\lambda$ can be chosen for each of the $\ell$ dotted boxes in $\mu^\circ$ and, at the end of the process, the dotted boxes are all in different columns of $\lambda$.

 Once $\ell$ boxes contain dots in $\lambda$, remove the dotted boxes from $\lambda$ to form $\mu \subseteq \nu$.  
Any box in $\nu/\mu$ must correspond to a box of one of the following three types: a box that belongs to $\nu/\mu^-$, a box that belongs to $\mu^\circ/\lambda$ but not $\mu^\circ/\mu^-$, or a box that is dotted in $\lambda$, via the above process, but not dotted in $\mu^\circ$.  Boxes in $\nu/\mu$ of the first two types together form a horizontal strip in $\nu/\mu$.  The rightmost box in $\nu/\mu$ of the third type  belongs to a column of $\nu$ that is strictly left of the leftmost column that contains a box from $\nu/\mu^-$.  As well, any box in $\nu/\mu$ of the third type corresponds to a dotted box in $\lambda$ that belongs to a column of $\nu$ that does not contain a box corresponding to the second type, since boxes in $\mu^\circ /\lambda$ form a horizontal strip.  It follows that $\nu/\mu$ is a horizontal strip and $(\lambda,\mu,\nu)$ is an $sp$-bumping sequence.

Given a pair $((\lambda,\mu,\nu),m_1)$ as in the statement of the lemma, we reverse the above procedure to produce $(\lambda,\mu^\circ,\mu^-,\nu)$.  If $\mu=\nu$, then $k=\ell$, so $m_1=m_2=0$.  In this case, $\mu^\circ=\lambda$ and $\mu^-=\nu$.  Otherwise, remove the rightmost $m_2$ boxes from $\nu$ that belong to $\nu$ but not to $\mu$ to give the shape $\mu^-$, where $\mu\subseteq \mu^-$.  If $m_2=k-\ell$ so that $\mu=\mu^-$, then $m_1=0$ so let $\mu^\circ=\lambda$.  Otherwise, place dots in all boxes in $\mu^-$ that belong to $\mu^-$ but not to $\mu$; there are $m_1$ such boxes.  Working from top to bottom and from left to right, if a dotted box does not belong to $\lambda$, add a box to 
$\lambda$ in that position.  If a dotted box does belong to $\lambda$, add a box to $\lambda$ in the first row below this box.  The new shape formed in this way is $\mu^\circ$, which contains $m_1$ more boxes than $\lambda$ and the sequence $(\lambda,\mu^\circ,\mu^-,\nu)$ is an $spo$-bumping sequence.  \end{proof}

\begin{exa} {\ }  Consider the $spo$-bumping sequence $(\lambda,\mu^\circ,\mu^-,\nu)$, where \\
$\lambda=\vy{ & & & \bullet \cr & & \cr & \bullet & \bullet \cr \cr \cr}$\ \ $\mu^\circ=\vy{ & & & \bullet \cr & & \cr & & \bullet \cr & \bullet \cr \cr \cr}$ \ \ $\mu^-=\vy{& &  \cr  & & \cr & \cr \cr \cr \cr}$ \ \ $\nu=\vy{ & & & & \cr & & \cr & \cr \cr \cr \cr}$

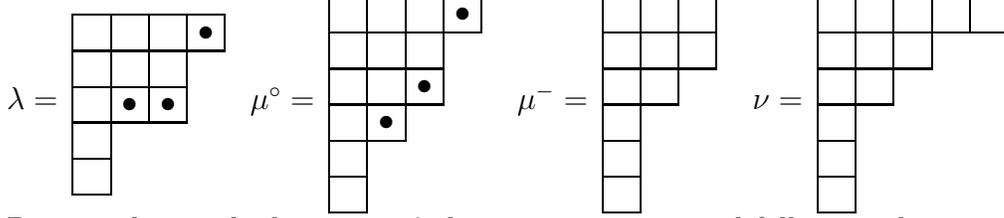
\n By considering the boxes in $\mu^\circ$ that are not in $\mu^-$ and following the procedure in the proof of Lemma \ref{bijectionlem}, with boxes dotted as above, we obtain $\mu=(3,3,1,1,1)$ 
and  $sp$-bumping sequence
$(\lambda, \mu, \nu)$.

\end{exa}

\begin{exa}
Let $m_1=m_2=2$ and $\ell=3$ and consider the $sp$-bumping sequence $(\lambda, \mu, \nu)$ below: \\

$\lambda=\vy{ & & & \cr & & \cr & & \cr \cr \cr}$ \ \ $\mu=\vy{ & & \cr  & & \cr \cr \cr \cr}$ \ \ $\nu=\vy{& & & & \cr & & \cr & \cr \cr \cr \cr}$

\noindent Then $\mu^-=\vy{& & \cr & & \cr & \bullet \cr \cr \cr  \bullet \cr}$  and $\mu^\circ=\vy{& & & \cr & & \cr & & \cr & \cr \cr \cr}$,
 where the dotted boxes are determined as in the proof of Lemma \ref{bijectionlem}.

\end{exa}

\begin{cor}Let $T$ be an $spo$-tableau of shape $\lambda$ and let $U$ be a one-row $(k)$-tableau.  Then $T\cdot U$ has shape $\nu$, where $\nu$ is the end shape in an $sp$-bumping sequence $(\lambda,\mu,\nu)$.  In other words, $\lambda/\mu$ is a horizontal strip with $\vert \lambda/\mu \vert=\ell$, for some $0 \leq \ell \leq k$,  and $\nu/\mu$ is a horizontal strip with $\vert \nu/\mu\vert=k-\ell$.

\end{cor}

\begin{thm} Consider an $sp$-bumping sequence $(\lambda,\mu,\nu)$, where $\lambda/\mu$ is a horizontal strip with $\vert \lambda/\mu \vert=\ell$, for some $0 \leq \ell \leq k$,  and $\nu/\mu$ is a horizontal strip with $\vert \nu/\mu\vert=k-\ell$.
If $S$ is an $spo$-tableau of shape $\nu$, then $S$ can be decomposed as a product of an $spo$-tableau of shape $\lambda$ and a $(k)$-tableau.

\end{thm}
\begin{proof}  We first associate a unique $m_1$ and $m_2$ to ($S$, $(\lambda,\mu,\nu)$), which will allow us to find the $spo$-bumping sequence corresponding to $((\lambda,\mu,\nu),m_1)$ in accordance with Lemma \ref{bijectionlem}.  We will then decompose $S$  by reversing the $spo$-insertion algorithm.  Since $\ell= \vert \lambda / \mu \vert$, $k-\ell= \vert \nu / \mu \vert$, and $m_1+m_2=k- \ell$, we need only determine a value for $m_1$ or $m_2$.  

If $S$ is a symplectic tableau (so does not contain entries from $B_1=\{1^\circ,\ldots,n^\circ\}$) then $m_1=0$ and if $\mu=\nu$ then $m_1=m_2=0$.  Otherwise, consider the rightmost box in $\nu$ that does not belong to $\mu$ and begin reversing the $spo$-insertion process with the entry $x$ that belongs to this box in $S$.  If $x \in B_1$ and is in the first column  of $S$ or if $x \in B_0=\{1,\ov{1},\ldots,m,\ov{m}\}$ and is in the first row, $x$ bumps out of the tableau and the box it occupies is removed from $S$, leaving a tableau $S_1$ of shape $\nu_1$ with one less box.   Otherwise, if  $x \in B_1$, bump it into the column immediately to its left by displacing the largest entry in that column that is less than $x$ and replacing it with $x$.  If $x \in B_0$, bump it into the row above it by replacing the largest entry in that row that is less than it.  Continue the procedure with the displaced entries until an entry bumps out of the first row or the first column, leaving a tableau $S_1$ of shape $\nu_1$ with one less box.  If  this results in an entry from $B_1$ getting  bumped out of the first column, then $m_2=0$.  Otherwise, repeat the procedure for the entry in the rightmost box in $S_1$ that does not belong to $\mu$  until an entry bumps out of the first row or column.  Repeat the procedure until $k-\ell$ entries have been bumped out of the tableau or until an entry in $B_1$ is bumped out of the first column.  Let $m_2$ equal  the number of entries in $B_0$ that are bumped out of the first row during this process.   The tableau $T_1$ remaining after bumping these $m_2$ entries from the tableau   will have shape $\mu^-$ in the $spo$-bumping sequence described below.   If $m_2=k-\ell$, then $\mu=\mu^-$ and $m_1=0$.  Otherwise, if  the rightmost box in $T_1$ that does not belong to $\mu$ is in column $j$, then the entry in this box belongs to $B_1$  and, since bumping this entry through the tableau causes an entry from $B_1$ to bump out of the first column,  for each column $i$ with $1 \leq i \leq j$ there is an entry $a_i \in B_1$ in column $i$ and an entry $a_{i-1} \in B_1$ in column $i-1$, with $a_{i-1}<a_i$, that belongs to a row below $a_i$. 

We can now find the image of  $((\lambda,\mu,\nu),m_1)$, using Lemma \ref{bijectionlem}, which is an $spo$-bumping sequence $(\lambda,\mu^\c, \mu^-,\nu)$.  We now verify that the $spo$-insertion process can be reversed, as described in the paragraphs following Theorem \ref{shapethm}, using this particular sequence and these values of $m_1$ and $m_2$.  Reversing the $spo$-insertion process first bumps the $m_2$ entries from $B_0$ as described above, leaving $T_1$, which has shape $\mu^-$.  

Continuing the reversed $spo$-insertion algorithm then adds $\ell$ empty boxes to $T_1$,  which are each moved through the tableau through the jeu de taquin process, resulting in a tableau $T_2$ of shape $\mu^\circ$.  Provided the empty boxes each start in rows below the first, the tableau $T_2$ still contains an entry from $B_1$ in each of columns $1$ through $j$ since the jeu de taquin process cannot move an entry $a_{i-1}$ below an entry $a_i$, where $a_{i-1}<a_i$.  If an empty box starts in the first row of $T_1$, then $\mu^-$ has fewer boxes in the first row than  $\mu^\circ$ so, by Lemma \ref{rowonelem}, $\lambda$ and $\mu^\circ$ have the same number of boxes in the first row, which means that no entries are bumped out of the first row in the next stage of the reverse $spo$-algorithm.

Since $\mu^\circ/\lambda$ is a horizontal strip with  $m_2$ boxes corresponding to columns left of column $j+1$ in $\mu^\circ$, any entry in $T_2$ that belongs to a box of  $\mu^\circ / \lambda$ must be at the bottom of column $j$ or at the bottom of a column left of column $j$, so contains an entry from $B_1$.  Thus the final stage of the $spo$-insertion algorithm can be completed to bump $m_2$ entries from  $B_1$ from $T_2$, producing a tableau $T$ of shape $\lambda$ and a $(k)$-tableau, whose product is $S$.
\end{proof}

\begin{exa}Consider the $spo$-tableau $\vy{2 & 1^\c & 3^\c \cr \ov{2}\cr 2^\c \cr}$ and the $sp$-bumping sequence\\
$\left(\vy{& & \cr \cr \cr}, \vy{ \cr \cr \cr}, \vy{ & & \cr \cr \cr }\right)$.    Reversing the $spo$-insertion process for the entry in the $(1,3)$-box  bumps a $2$ out of the first row, and then repeating the process for the entry in the $(1,2)$-box in the resulting tableau bumps a $2^\circ$ out of the first column.  It follows that $m_2=1$ and, since $\ell=2$, $m_1=1$.  The corresponding $spo$-bumping sequence, via Lemma \ref{bijectionlem}, is then 
 $\left(\vy{& & \cr \cr \cr}, \vy{& & \cr & \cr \cr}, \vy{ & \cr \cr \cr}, \vy{ & & \cr \cr \cr }\right)$.  We can use this $spo$-bumping sequence and the given tableau to reverse the $spo$-insertion process and decompose the tableau as a product:
 $$\vy{2 & 1^\c & 3^\c \cr \ov{2}\cr 2^\c \cr}= \vy{\ov{1} & \ov{1} & \ov{2} \cr 1^\circ  \cr 3^\circ\cr } \cdot \vy{1 & 1 & 2 & 2^\circ \cr}.$$
 
 \noindent Note that if we consider the same $spo$-tableau but  $sp$-bumping sequence \\
 $\left(\vy{& & \cr \cr \cr}, \vy{ & \cr \cr}, \vy{ & & \cr \cr \cr }\right)$, then the corresponding $spo$-bumping sequence is\\
 $\left(\vy{& & \cr \cr \cr}, \vy{& &  \cr \cr  \cr \cr}, \vy{ & \cr \cr \cr}, \vy{ & & \cr \cr \cr }\right)$, which yields
  $$\vy{2 & 1^\c & 3^\c \cr \ov{2}\cr 2^\c \cr}= \vy{\ov{1} & \ov{1} & 3^\circ \cr \ov{2} \cr 1^\circ  \cr } \cdot \vy{1 & 1 &  2 & 2^\circ \cr}.$$
\end{exa}

\bigskip

Theorem \ref{mainthm} now follows.  Our final example illustrates the orthosymplectic Pieri rule.

\begin{exa} Let $\lambda=(3,1)$ and $k=3$. Then $\ell=0, 1, 2$ or $3$.  We have
$$spo_\lambda spo_{(k)}=spo_{(6,1)}+ spo_{(5,2)}+spo_{(5,1,1)}+spo_{(4,3)}+spo_{(4,2,1)}+2spo_{(4,1)}+2spo_{(3,2)}$$
$$+spo_{(3,1,1)}+spo_{(2,2,1)}+
spo_{(5)}+2spo_{(2,1)}+spo_{(1,1,1)}+spo_{(3)}+spo_{(1)}.$$

\end{exa}

\subsection*{Acknowledgement}
{The author wishes to thank an anonymous referee, who carefully read the paper and provided useful suggestions.}

\end{document}